\documentclass[12pt,english,a4paper,twoside]{article}
\usepackage[centertags]{amsmath}
\usepackage{amssymb}
\usepackage{amsthm}
\usepackage{makeidx}
\usepackage{babel,a4wide}

\newtheorem{theorem}{Theorem}[section]

\newtheorem{lemma}[theorem]{Lemma}
\newtheorem{proposition}[theorem]{Proposition}
\newtheorem{definition}[theorem]{Definition}
\newtheorem{remark}[theorem]{Remark}
 \numberwithin{equation}{section}

\begin{document}

\title{Appell polynomial sequences with respect to some differential operators}

\author{P. Maroni, T. A. Mesquita\footnote{Corresponding author (teresam@portugalmail.pt)}}
\date{}
\maketitle

\begin{center}
CNRS, UMR 7598, Laboratoire Jacques-Louis Lions, F-75005, Paris, France \& \\
UPMC Univ Paris 06, UMR 7598, Lab. Jacques-Louis Lions,
F-75005, Paris, France\\
maroni@ann.jussieu.fr

Instituto Superior Polit\'ecnico de Viana do Castelo \& CMUP\\
Av. do Atl\^antico, 4900-348 Viana do Castelo, Portugal\\
teresam@portugalmail.pt 

\end{center}



\begin{abstract}
We present a study of a specific kind of lowering operator, herein called $\Lambda$, which is defined as a finite sum of lowering operators, proving that this configuration can be altered, for instance, by the use of Stirling numbers.  We characterize the polynomial sequences fulfilling an Appell relation with respect to $\Lambda$, and considering a concrete cubic decomposition of  a simple Appell sequence, we prove that the polynomial component sequences are $\Lambda$-Appell, with $\Lambda$ defined as previously, although by a three term sum. Ultimately, we prove the non-existence of orthogonal polynomial sequences which are also $\Lambda$-Appell, when $\Lambda$ is the lowering operator $\Lambda=a_{0}D+a_{1}DxD+a_{2}\left(Dx\right)^2D$, where $a_{0}$, $a_{1}$ and $a_{2}$ are constants and $a_{2} \neq 0$. The case where $a_{2}=0$ and $a_{1} \neq 0$ is also naturally recaptured.
\end{abstract}

\section*{Introduction}\label{sec:introduction}

It is a well know fact that the only monic orthogonal polynomial sequence ${\{B_{n}\}}_{n \geq 0}$ satisfying the relation $DB_{n+1}(x)$ $=$ $(n+1)B_{n}(x),$ $\; n \geq 0$, for the ordinary derivative operator $D$, is the Hermite sequence, up to an affine transformation \cite{Chihara}. This last relation defines the so-called Appell sequences \cite{Appell} which are widely spread in literature, in several contexts and applications. They present a large variety of features and include other famous polynomial sequences like the Bernoulli sequence. To this matter we can consult, for instance, \cite{Quad-Appell,Interpolation, Douak-Appell-d-ortho}, among many others.\\
During an investigation based on a cubic decomposition of an Appell sequence, we find polynomial sequences fulfilling the analogous identity: 
\newline $\Lambda B_{n+1}(x)$ $=$ $\rho_{n}B_{n}(x), \; n \geq 0$, for certain lowering operators $\Lambda$, where $\rho_{n}$ are normalization constants.
The operators emerging from the context of that cubic decomposition, hold the multiplicative form  $\big(2I+3xD\big)\,D\,\big(I+3xD\big)$ or $\big(I+3xD\big)\,\big(2I+3xD\big)\,D$, that, as it is explained hereafter, it can be written as $\Lambda=a_{0}D+a_{1}DxD+a_{2}\left(Dx\right)^2D$, with constant coefficients $a_{0}$, $a_{1}$ and $a_{2}$. Indeed, in the totality of results of this work, we use different ways of expressing the same operator, and itself gained a remarkable role along the way, requiring a treatment without particular restrictions. Different versions of these operators appear often in the study of special functions, for instance, with regard to the monomiality principle (see, for example, \cite{Khan}--\cite{Srivastava-Ben cheik}).\\
For this reason, this manuscript is organized as follows. In the first section, the basic definitions and tools are given, and the second section is devoted to the study of the operator $\Lambda=\sum_{i=0}^{k}a_{i} \left(Dx\right)^{i} D$, where $k$ is a positive integer and $a_{k} \neq 0$, and to a characterization of the sequences herein called $\Lambda$-Appell. Summed up, the two first sections establish the basic ground of the upcoming results. In the third part, we consider a cubic decomposition of an Appell sequence and it is indicated what kind of Appell behaviour have the respective component sequences. The last section clarifies that we cannot find orthogonal sequences in the set of $\Lambda$-Appell sequences, for $\Lambda=a_{0}D+a_{1}DxD+a_{2}\left(Dx\right)^2D$, $a_{2} \neq 0$. This last result allow us to reformulate the case where $\Lambda=a_{0}D+a_{1}DxD$, provided earlier in \cite{QuadraticAppell} for $a_{0}= \epsilon$ and $a_{1}= 2$.

\section{Preliminaries}\label{sec:preliminaries}

\subsection{Basic definitions and notation}

Let $\mathcal{P}$ denote the vector space  of polynomials with coefficients in  $\mathbb{C}$ and let $\mathcal{P}'$ be its dual. We indicate by $\langle u,p \rangle$ the action of the form or linear functional  $u \in \mathcal{P}'$ on $p \in \mathcal{P}$.
In particular, $(u)_{n}= \langle u,x^{n} \rangle,\: n \geq 0$, are called the moments of $u$. A form $u$ is equivalent to the numerical sequence  $\{(u)_{n}\}_{n \geq 0}$.

In the sequel, we will call polynomial sequence (PS) to any sequence 
${\{B_{n}\}}_{n \geq 0}$ such that $\deg B_{n}= n,\; \forall n \geq 0$.
We will also call monic polynomial sequence (MPS) a PS such that in each polynomial the leading coefficient is equal to one. Notice that if $<u,B_{n}>=0,\; \forall n \geq 0$, then $u=0$. Given a MPS ${\{B_{n}\}}_{n \geq 0}$, there are complex sequences, ${\{\beta_{n}\}}_{n \geq 0}$ and
$\{\chi_{n,\nu}\}_{0 \leq \nu \leq n,\; n \geq 0},$ such that
\begin{align}
&B_{0}(x)=1, \;\; B_{1}(x)=x-\beta_{0},\label{divisao_ci}\\
&B_{n+2}(x)=(x-\beta_{n+1})B_{n+1}(x)-\sum_{\nu=0}^{n}\chi_{n,\nu}B_{\nu}(x).\label{divisao}
\end{align}
This relation is called the structure relation of  ${\{B_{n}\}}_{n \geq 0}$, and ${\{\beta_{n}\}}_{n \geq 0}$ and
$\{\chi_{n,\nu}\}_{0 \leq \nu \leq n,\; n \geq 0}$ are called the structure coefficients.
Moreover, there exists a unique sequence  ${\{u_{n}\}}_{n \geq 0},\;\;
u_{n} \in \mathcal{P}'$, called the dual sequence of  ${\{B_{n}\}}_{n \geq 0}$, such that
$$\langle u_{n},B_{m} \rangle= \delta_{n,m},\;\; n,m \geq 0,$$
where $ \delta_{n,m}$ denotes the Kronecker symbol. Let us remark that, if $p$ is a polynomial and $ \langle u_{n},p \rangle=0,\; \forall n \geq 0$, then $p=0$.
Besides, it is well known that, \cite{variations}
\begin{align}
&\label{betas}\beta_{n}= \langle u_{n},xB_{n}(x) \rangle,\:\: n \geq 0,\\
&\label{quisnus}\chi_{n,\nu}= \langle u_{\nu},xB_{n+1}(x) \rangle,\:\: 0\leq \nu \leq n,\:\: n \geq 0.
\end{align}

\begin{lemma}\cite{variations}\label{lema1}
For each $u \in \mathcal{P}^{\prime}$ and each $m \geq 1$, the two following propositions are equivalent.
\begin{description}
\item[a]  $\langle u, B_{m-1} \rangle \neq 0,\:\: \langle u, B_{n} \rangle =0,\: n \geq m$.
\item[b] $\exists \lambda_{\nu} \in \mathbb{C},\:\: 0 \leq \nu \leq m-1,\:\: \lambda_{m-1}\neq 0$ such that
$u=\displaystyle\sum_{\nu=0}^{m-1}\lambda_{\nu} u_{\nu}$, with $\lambda_{\nu} = \langle u, B_{\nu} \rangle.$ \end{description}
\end{lemma}

A linear operator $T:\mathcal{P} \rightarrow \mathcal{P}$ has a transpose  $^{\textrm{t}}T:\mathcal{P}^{\prime} \rightarrow \mathcal{P}^{\prime}$ defined by
$$\langle ^{\textrm{t}}T(u), p \rangle =  \langle u, T(p) \rangle,\hspace{0.4cm} u \in \mathcal{P}^{\prime} , \:p \in \mathcal{P}.$$

Therefore, given $\varpi \in \mathcal{P}$ and $u \in \mathcal{P}'$, the form $\varpi u$, called the
left-multiplication of $u$ by the polynomial $\varpi$, is defined by $\langle \varpi u,p \rangle= \langle u,\varpi p \rangle,\:\:\: \forall p \in \mathcal{P},$ and the transpose of the derivative operator on $\mathcal{P}$ defined by $p \rightarrow (Dp)(x)=p^{\prime}(x),$ is the following (see \cite{theoriealgebrique}).
\begin{equation}\label{funcionalDu}
u \rightarrow Du:\;\;\langle Du,p \rangle=-\langle u,p^{\prime} \rangle ,\:\:\: \forall p \in \mathcal{P}.\end{equation}
Hence, it is easily established that 
\begin{align}
\label{derivada do produto} &D(pu)=p^{\prime}u+pD(u);\\
\label{DpD^k} &D\big(p D^k\big)=p^{\prime}D^k+pD^{k+1}.
\end{align}

\begin{definition} \cite{theoriealgebrique,euler} \label{orthogonality definition}
A PS ${\{B_{n}\}}_{n \geq 0}$ is regularly orthogonal with respect to the form $u$ if and only if it fulfils 
\begin{align}
\label{ortogonal} &\langle u,B_{n}B_{m} \rangle=0,\:\: n \neq m,\:\:\:\:\: n, m \geq 0,\\
\label{ortogonal regular} & \langle u,B_{n}^{2} \rangle \neq 0, \: n \geq 0.
\end{align}
Then, the form $u$ is said to be regular (or quasi-definite) and ${\{B_{n}\}}_{n \geq 0}$ is an orthogonal polynomial sequence (OPS). The conditions (\ref{ortogonal}) are called the orthogonality conditions and the conditions (\ref{ortogonal regular}) are called the regularity conditions.
\end{definition}
We can normalize ${\{B_{n}\}}_{n \geq 0}$ in order that it becomes monic; then it is unique and we note it as a MOPS. Considering ${\{u_{n}\}}_{n \geq 0}$ the corresponding dual sequence, it holds $u=\lambda u_{0},$ with $\lambda=(u)_{0} \neq 0$.
\begin{lemma}\cite{euler}\label{phi-u=0}
Let $u$ be a regular form and $\phi$ a polynomial such that $\phi u=0$. Then $\phi=0$.
\end{lemma}
\begin{theorem}\cite{variations}\label{regular}
Let ${\{B_{n}\}}_{n \geq 0}$ be a MPS and ${\{u_{n}\}}_{n \geq 0}$ its dual sequence. 
The following statements are equivalent:
\begin{description}
  \item[a]  The sequence ${\{B_{n}\}}_{n \geq 0}$ is orthogonal (with respect to $u_{0}$);
  \item[b] $\chi_{n,k}=0$,  $\;0\leq k \leq n-1,\;\;\;\; n\geq 1; \;\;\;\chi_{n,n}\neq 0,\; \; n \geq
  0$;
  \item[c]$xu_{n}=u_{n-1}+\beta_{n}u_{n}+\chi_{n,n}u_{n+1}$, $\chi_{n,n}\neq
  0,\;\;\; n \geq 0$, $u_{-1}=0$;
  \item[d] For each $n \geq 0$, there is a polynomial  $\phi_{n}$ with $\deg(\phi_{n})=n$ such that $u_{n}=\phi_{n}u_{0}$;
  \item[e] $u_{n}=\Big(<u_{0},B_{n}^{2}>\Big)^{-1}B_{n}u_{0}$, $n \geq 0$;
\end{description}
where $\beta_{n}$ and $\chi_{n,k}$ are defined by (\ref{betas}-\ref{quisnus}).
\end{theorem}
Let ${\{B_{n}\}}_{n \geq 0}$ be a MOPS. From statement $b)$ of Theorem \ref{regular}, the structure relation (\ref{divisao})
becomes the following second order recurrence relation: 
\begin{align}\label{recurrencia de ordem dois}
&B_{0}(x)=1,\;\;B_{1}(x)=x-\beta_{0},\\
&B_{n+2}(x)=(x-\beta_{n+1})B_{n+1}(x)-\gamma_{n+1}B_{n}(x),\;\;\; n \geq 0,
\end{align}
where $\gamma_{n+1}=\chi_{n,n} \neq 0,\;\; n \geq 0$, and also by item $e)$, we have:
\begin{equation}\label{betas e gammas}
\beta_{n}=\frac{\langle u_{0}, xB_{n}^{2}(x) \rangle}{ \langle u_{0}, B_{n}^{2}(x) \rangle},\;\;\: \:\:\:\gamma_{n+1}=\frac{ \langle u_{0}, B_{n+1}^{2}(x) \rangle}{ \langle u_{0}, B_{n}^{2}(x) \rangle}.
\end{equation}
Note that the regularity conditions (\ref{ortogonal regular}) are fulfilled if and only if $\gamma_{n+1} \neq 0,\:\: n \geq 0$.
\medskip
\newline
Finally, we recall that a MPS ${\{B_{n}\}}_{n \geq 0}$ is called classical, if and only if it satisfies the Hahn\'{}s property \cite{Hahn}, that is to say, the derivative MPS ${\{B^{[1]}_{n}\}}_{n \geq 0}$, $B^{[1]}_{n}(x):=(n+1)^{-1}DB_{n+1}(x)$, is also orthogonal. The classical polynomials are divided in four classes: Hermite, Laguerre, Bessel and Jacobi \cite{Chihara,euler}, and characterized by the functional equation
$$D(\phi u)+\psi u =0,$$
where $\psi$ and $\phi$ are two polynomials such that: $\deg \psi =1$, $\deg \phi \leq 2$ and $\psi^{\prime}-\frac{1}{2}\phi^{\prime\prime}n \neq 0,\: n \geq 1$. For example, the polynomials $\phi(x)=x$ and $\psi(x)=x-\alpha-1$, with parameter $\alpha \notin \mathbb{Z}^{-}$, correspond to the Laguerre polynomials.


\subsection{Lowering operators and Appell sequences}

An Appell sequence $\{B_{n}\}_{n \geq 0}$ is usually defined by the condition \cite{Appell} 
$$\displaystyle B^{[1]}_{n}(x) = \frac{1}{n+1}DB_{n+1}(x) =B_{n}(x), \; n \geq 0.$$
The derivation operator is an example of an usually called lowering operator, that is to say,  is a linear mapping $\mathcal{O}: \mathcal{P} \rightarrow \mathcal{P}$ fulfilling the two conditions: 
\newline\begin{center}
$\mathcal{O}(1)=0$ and $\deg\left(\mathcal{O}\left(x^{n}\right)\right)=n-1,\; n \geq 1$. 
\end{center}
More generally, for any lowering operator $\mathcal{O}$, we can construct the sequence  $\{B^{\left[1\right]}_{n}\left(\cdot\,; \mathcal{O} \right)_{n}\}_{n \geq 0}$, defined by 
$$B^{\left[1\right]}_{n}\left(x; \mathcal{O} \right)=\rho_{n}^{-1} \left( \mathcal{O}B_{n+1}\right)(x),\; n \geq 0,$$
where $\rho_{n} \in \mathbb{C}\backslash \{0\}$ is chosen in order to make $B^{\left[1\right]}_{n}\left(x; \mathcal{O} \right)$ monic, and attend to the next definition.
\begin{definition}\cite{ Ben cheik monomiality 2, Ben cheik - Chaggara}
A MPS $\{B_{n}\}_{n \geq 0}$ is called an $\mathcal{O}$-Appell sequence with respect to a lowering operator $\mathcal{O}$ if $B_{n}(\cdot)=B_{n}^{[1]}\left(\cdot\, ; \mathcal{O}\right)$ for all integers $n \geq 0$.
\end{definition}
In a forthcoming section, it will be important to understand the application of a lowering operator to a sequence of polynomials $\{\zeta_{n}\}_{n \geq 0}$ which does not necessarily fulfils all of the MPS attributes, as for instance, when $\deg \zeta_{n} < n$ for some values of $n$. In that situation, we  can guarantee a structure to those sequences when some hypotheses are taken, as the next Proposition \ref{O-Appell sequences structure} announces. In its proof, it will be useful the following trivial result.
\begin{lemma}\label{basic lemma lower operator}
Let $\mathcal{O}$ be a lowering operator and  $f \in \mathcal{P}$. If $\deg\left( f \right)>0$, then $\deg\left(\mathcal{O}\left( f \right)\right)=\deg\left( f \right)-1$.
\end{lemma}
\begin{proof}
Let us consider $f \in \mathcal{P}$ such that $\deg\left( f \right)>0$. Then, $f=\displaystyle \sum_{i=0}^{k} a_{i}x^{i}$ with $a_{k} \neq 0$, $k \geq 1$, and
$$\mathcal{O}\left( f \right)=  \displaystyle \sum_{i=0}^{k} a_{i}\mathcal{O} \left(x^{i} \right)= \displaystyle \sum_{i=1}^{k} a_{i} \left(   \displaystyle \sum_{j=0}^{i-1}  \theta_{i,j} x^{j}\right)=\displaystyle \sum_{j=0}^{k-1} \left(   \displaystyle \sum_{i=j+1}^{k}  a_{i} \theta_{i,j}\right)  x^{j}.$$
In particular, the coefficient of $x^{k-1}$ is $a_{k}\theta_{k,k-1} \neq 0$.
\end{proof}
\begin{proposition}\label{O-Appell sequences structure}
Let $\mathcal{O}$ be a lowering operator and let $\{ f_{n} \}_{n \geq 0}$ be a sequence in $\mathcal{P}$ such that $\rho_{n} f_{n}=\mathcal{O}\left(f_{n+1}\right)$, with $\rho_{n} \in \mathbb{C}\backslash \{0\}$. Then, either $f_{n}=0,\, n \geq 0$, or there is $n_{0} \geq 0$ such that $f_{n_{0}} \neq 0$, $f_{i}=0, \, 0 \leq i \leq n_{0}-1$, if $n_{0} \geq 1$, and $\deg \left(f_{n+1} \right)=\deg \left(f_{n} \right)+1,\, n \geq n_{0}$.
\end{proposition}
\begin{proof}
Let $n_{0}$ be the smallest index of a nonzero element of $\{ f_{n} \}_{n \geq 0}$, that is, $f_{n_{0}} \neq 0$ and $f_{i}=0, \,0 \leq i \leq n_{0}-1$, if $n_{0} \geq 1$.  Then, $\rho_{n_{0}} f_{n_{0}}= \mathcal{O}\left(f_{n_{0}+1}\right)$. 
Obviously,  $\deg \left( f_{n_{0}+1} \right) > 0$, otherwise we would have $f_{n_{0}} =0$.
Thus, by Lemma \ref{basic lemma lower operator}, $\deg \left( f_{n_{0}} \right)=\deg \left( \mathcal{O} \left( f_{n_{0}+1} \right)\right)= \deg \left( f_{n_{0}+1} \right)-1$ and $\deg \left( f_{n_{0}+1} \right)=\deg \left( f_{n_{0}} \right)+1$.
\newline
By finite induction, using the same arguments, we can easily prove that $\deg \left( f_{n+1} \right)=\deg \left( f_{n} \right)+1, \; n \geq n_{0}$, which concludes the proof.
\end{proof}

\medskip

Let us fix a non-negative integer $k$ and some constants $a_{i} \in   \mathbb{C},\; i=0,\ldots k,$ and let us set $\Lambda=\displaystyle\sum_{i=0}^{k}a_{i}\left(Dx\right)^{i}D$. 


\begin{lemma}
The following identity holds, for any positive integer $n$,
\begin{equation}\label{lambda x^n}
\Lambda\left(x^{n}\right)=\left( n \displaystyle\sum_{i=0}^{k} a_{i}n^{i} \right) x^{n-1}.\\
\end{equation}
\end{lemma}
\begin{proof} It is easy to see that 
\begin{equation}\label{DxD-xn}
\left(Dx\right)^{i}D\left( x^{n} \right)=D\left(xD\right)^{i}\left( x^{n} \right)=n^{i+1}x^{n-1},\; i \geq 0,\end{equation}
whence the desired result.
\end{proof}

As a consequence, we have the next Proposition.
\begin{proposition}\label{lambda lowering operator}
\begin{description}
\item[i ] For any positive integer $k$,  $\Lambda =0$ $\Leftrightarrow$ $a_{i}=0,\;\; i =0,\ldots, k$.
\item[ii] Suppose $a_{k} \neq 0$, then the operator $\Lambda$ is a lowering operator if and only if the polynomial $\displaystyle f(x)=\sum _{i=0}^{k}a_{i}x^{i}$ has no positive integer root.
\end{description}
\end{proposition}
\begin{proof}
$i)$ Suppose $\Lambda =0$, that is to say, $\Lambda\left( f \right) = 0$ for any $f \in \mathcal{P}$. In particular, $0 = \Lambda\left(x^{n}\right)=\left( n \displaystyle\sum_{i=0}^{k} a_{i}n^{i} \right) x^{n-1}$. Therefore, $\displaystyle\sum_{i=0}^{k} a_{i}n^{i} =0$, for all positive integer $n$, which implies $a_{i}=0,\;\; i =0,\ldots, k$. The reciprocity is obvious.
\newline $ii)$ Since $a_{k} \neq 0$, the polynomial $\displaystyle f(x)=\sum _{i=0}^{k}a_{i}x^{i}$ is not identically zero and $\Lambda \neq 0$. Moreover, $\Lambda$ is a lowering operator by virtue of (\ref{lambda x^n}) and the assumption.
\end{proof}


\section{The $\Lambda$ operator}

\subsection{Further definitions and the transpose operator}

The form used in the definition of the operator $\Lambda$ might seem very limitative, but in fact, as we will see right ahead, this differential operator can be presented in different layouts. 
In particular, the next results clarify that a $\Lambda$ operator (for all non-negative integer $k$) can be expressed as a product of simpler operators or as a linear combination of operators of the form $x^{i}D^{i+1}$, with $i=0,\ldots,k$. The first lemma is a simple case of the problem of normal ordering of words in $D$ and $x$, where the Stirling numbers and their generalizations have a major role. A  comprehensively study of this problem can be found in the literature, namely in \cite{Blasiak, Mansour} and in the references therein. In those studies, it is considered the partial commutation relation between two operators $U$ and $V$ defined by $UV-VU=1$. With regard to the operators $x$ and $D$, it will be useful later-on to recall that
$\left(Dx\right)D=D+xD^{2},$ or 
\begin{equation}\label{xD^2=DxD-D}
xD^2=DxD-D.
\end{equation}

\begin{lemma}\label{power(Dx)D}
For all non-negative integer $i$, it holds
\begin{equation} \label{Dx-to-powers}
\left(Dx\right)^{i}D=\sum_{m=0}^{i} S\left( i+1, m+1 \right)  x^{m}D^{m+1},
 \end{equation}
where $S\left(n,k\right)$ are the Stirling numbers of second kind.
Conversely, \begin{equation} \label{powers-to-Dx}
x^{i}D^{i+1}=\sum_{m=0}^{i}s\left(i+1,m+1\right)\left(Dx\right)^{m}D,
\end{equation}
where $s\left(n,k\right)$ are the Stirling numbers of first kind.
\end{lemma}
\begin{proof}
Identity (\ref{Dx-to-powers}) can be consulted in \cite{Comtet} (p.220), and the reciprocal statement is justified by the orthogonality $\sum S\left(n,k \right)s\left( k,m\right) = \delta_{n,m}$, showing that each of the relations above implies the other \cite{Riordan}.
\end{proof}
As a corollary of Lemma \ref{power(Dx)D}, we can enunciate the identities of the next Lemma.
\begin{lemma}\label{linear combination of power(Dx)D}
\begin{equation}
\sum_{i=0}^{k}a_{i}\left(Dx\right)^{i}D =\sum_{m=0}^{k} \left( \sum_{i=m}^{k}   a_{i}  S\left( i+1, m+1 \right) \right) x^{m}D^{m+1};\end{equation}
\begin{equation} \sum_{i=0}^{k}a_{i}x^{i}D^{i+1} = \sum_{m=0}^{k} \left( \sum_{i=m}^{k}   a_{i}  s\left( i+1, m+1 \right) \right) \left(Dx\right)^{m}D; \end{equation}
\begin{equation}\label{lambda withk=2 in xD}
a_{0}D+a_{1}DxD+a_{2}\left(Dx\right)^{2}D= \left( a_{0}+a_{1}+a_{2} \right)D+ \left(a_{1}+3a_{2} \right)xD^{2}+a_{2}x^{2}D^{3}.
\end{equation}
\end{lemma}
\begin{proof}
A straightforward calculation yields the result. A table with the first Stirling numbers of second and first kind can be found, for instance, in \cite{Comtet}.
\end{proof}
\begin{proposition}\label{Lambda-as-a-product}
Let us consider two non-negative integers $l$ and $t$ and the operator
$$S_{l,t}=\prod_{-l}^{0}(A_{i}I+B_{i}xD)\,D\, \prod_{0}^{t}(A_{i}I+B_{i}xD),$$
where $I$ denotes the identity operator in $\mathcal{P}$ and $A_{0}=1, \, B_{0}=0, \, A_{i}, \, B_{i} \in \mathbb{C}, i=-l,\ldots, t$.
Then, for some contants $a_{i} \in \mathbb{C}, \; i=0, \ldots, l+t$, we have: 
\begin{equation}\label{operatorSlt2}
S_{l,t}=\sum_{i=0}^{l+t}a_{i}\left(Dx\right)^{i}D.
\end{equation}
Conversely, an operator $\Lambda=\displaystyle\sum_{i=0}^{k}a_{i}\left(Dx\right)^{i}D=\displaystyle\sum_{i=0}^{k}a_{i}D\left(xD\right)^{i}$ can be written in the form
$$\Lambda=S_{0,k}=D\, \prod_{1}^{k}(A_{i}I+B_{i}xD),$$
where coefficients $A_{i}$ and $B_{i}$ are obtained through the factorization of polynomial $f(x)= \sum_{i=0}^{k}a_{i}x^{i}$, that is, $\displaystyle\prod_{1}^{k}(A_{i}+B_{i}x)=\sum_{i=0}^{k}a_{i}x^{i}$.
\end{proposition}

\begin{proof}
Proceeding by induction, we see that 
\newline $S_{0,0}=D=(Dx)^{0}D$;
\newline $S_{1,0}=\displaystyle\sum_{i=0}^{1}a_{i}(Dx)^{i}D$, with $a_{0}=A_{-1}-B_{-1}$ and $a_{1}=B_{-1}$;
\newline and $S_{0,1}=\displaystyle\sum_{i=0}^{1}a_{i}(Dx)^{i}D$, with $a_{0}=A_{1}$ and $a_{1}=B_{1}$.
\newline
Let us assume that for some $l$ and $t$ we have $S_{l,t}=\displaystyle\sum_{i=0}^{l+t}a_{i}\left(Dx\right)^{i}D.$ Taking into account that $S_{l+1,t}=(A_{-l-1}I+B_{-l-1}xD)S_{l,t}$, we have:
\begin{align*}
S_{l+1,t}=&\left(A_{-l-1}I+B_{-l-1}xD\right) \sum_{i=0}^{l+t}a_{i}\left(Dx\right)^{i}D\\
=& B_{-l-1}a_{0}xD^{2}+\sum_{i=0}^{l+t}A_{-l-1}a_{i}\left(Dx\right)^{i}D+\sum_{i=1}^{l+t}B_{-l-1}a_{i}xD^{2}x\left(Dx\right)^{i-1}D.
\end{align*}
Considering identity (\ref{xD^2=DxD-D}), we obtain:
\begin{align*}
S_{l+1,t}=&\left(A_{-l-1}-B_{-l-1}\right)a_{0}D+B_{-l-1}a_{0}DxD\\
 & + \sum_{i=1}^{l+t}\left(A_{-l-1}-B_{-l-1}\right)a_{i}\left(Dx\right)^{i}D+\sum_{i=2}^{l+t+1}B_{-l-1}a_{i-1}\left(Dx\right)^{i}D\\
=&\sum_{i=0}^{l+t+1}\tilde{a}_{i}\left(Dx\right)^{i}D,
\end{align*}
where $\tilde{a}_{0}=\left(A_{-l-1}-B_{-l-1}\right)a_{0}$;   $\tilde{a}_{i}=B_{-l-1}a_{i-1}+\left(A_{-l-1}-B_{-l-1}\right)a_{i}$, if $i=1,\ldots, l+t$, and $\tilde{a}_{l+t+1}=B_{-l-1}a_{l+t}$.
\newline
Similarly, we can see that $S_{l,t+1}=\displaystyle \sum_{i=0}^{l+t+1}\tilde{a}_{i}\left(Dx\right)^{i}D$, where $\tilde{a}_{0}=A_{t+1}a_{0}$; $\tilde{a}_{i}=A_{t+1}a_{i}+B_{t+1}a_{i-1}$, if $i=1,\ldots, l+t$, and $\tilde{a}_{l+t+1}=B_{t+1}a_{l+t}$.
\newline The converse statement related to the operator $\Lambda$ is self-explanatory.
\end{proof}
%
In this manner, we have established that the three operators $\displaystyle\sum_{i=0}^{k}a_{i}\left(Dx\right)^{i}D$, $\displaystyle\sum_{i=0}^{k}\tilde{a}_{i}x^{i}D^{i+1}$ and 
$S_{l,t}=\displaystyle\prod_{-l}^{0}(A_{i}I+B_{i}xD)\,D\, \prod_{0}^{t}(A_{i}I+B_{i}xD)$ represent the same kind of operators, and consequently, we can select the more suitable form for each step of our path.
\begin{lemma}\label{The transpose operator tLambda}
The transpose operator $^{t}\Lambda: \mathcal{P}^{\prime} \rightarrow \mathcal{P}^{\prime}$ is defined by
\begin{equation}\label{lambda transpose}
^{t}\Lambda=\displaystyle\sum_{i=0}^{k}a_{i}(-1)^{i+1}\left(Dx\right)^{i}D.
\end{equation}
\end{lemma}
\begin{proof}
By definition, we have:
\begin{align*}
\left< ^{t}\Lambda u,f \right>=\left< u, \Lambda f \right>=\sum_{i=0}^{k}a_{i} \left< u, \left(Dx\right)^{i}D f \right>.\\
\end{align*}
Also, we can prove, by induction over $i$, regarding (\ref{funcionalDu}), that 
$$ \left< u, \left(Dx\right)^{i}D f \right>=(-1)^{i+1}\left< \left(Dx\right)^{i}D u, f \right>, \; i \geq 0.$$
Therefore,
$$\left< ^{t}\Lambda u,f \right>=\sum_{i=0}^{k}a_{i}(-1)^{i+1}\left< \left(Dx\right)^{i}D u, f \right>=\left<\sum_{i=0}^{k}a_{i}(-1)^{i+1} \left(Dx\right)^{i}D u, f \right>.\eqno\qedhere$$
\end{proof}

\subsection{$\Lambda$-Appell sequences}

Let us suppose that the fixed operator $\Lambda=\displaystyle\sum_{i=0}^{k}a_{i}\left(Dx\right)^{i}D$ is a lowering operator according to Proposition \ref{lambda lowering operator}.
Attending to the following 
\begin{align*}
\left( \Lambda B_{n+1}\right) (x)= \Lambda \left(x^{n+1}+\cdots \right)= (n+1)\left( \sum_{i=0}^{k} a_{i}(n+1)^{i} \right)x^{n} +\Lambda \left(\cdots \right)
\end{align*}
and recalling that $ \sum_{i=0}^{k} a_{i}(n+1)^{i} \neq 0$, we have:
\begin{equation}\label{Bn[1] definition}
B^{\left[1\right]}_{n}\left(x; \Lambda \right)=(n+1)^{-1}\left( \sum_{i=0}^{k} a_{i}(n+1)^{i}\right)^{-1} \left( \Lambda B_{n+1}\right) (x).
\end{equation}
Denoting by $\{u^{[1]}_{n} \left( \Lambda\right)\}_{n \geq 0}$ the dual sequence of $\{B_{n}^{\left[1\right]}\left(\cdot\,; \Lambda \right)_{n}\}_{n \geq 0}$, we have the following result.
\begin{proposition}\label{tLambda an[1]}
\begin{equation}\label{lambda transpose un}
 ^{t}\Lambda\left(u^{[1]}_{n} \left( \Lambda\right)\right)=\rho_{n}u_{n+1},
 \end{equation}
where $ ^{t}\Lambda$ is defined by (\ref{lambda transpose}) and $  \rho_{n}=(n+1)\left(\displaystyle \sum_{i=0}^{k}a_{i}(n+1)^{i}\right)$.
\end{proposition}
\begin{proof}
\begin{align*}
& \left<  u^{[1]}_{n}\left( \Lambda\right),  B_{m}^{\left[1\right]}\left(x; \Lambda \right)   \right>= \delta_{n,m}, \; n, m \geq 0.\\
& \textrm{That\, is,}\,  \left<  u^{[1]}_{n}\left( \Lambda\right),  \Lambda B_{m+1}(x) \right>=\rho_{m} \delta_{n,m},\\
& \textrm{or,}\,  \left<   ^{t}\Lambda \left(u^{[1]}_{n}\left( \Lambda\right)\right),  B_{m+1}(x) \right>=\rho_{m}\delta_{n,m}.
\end{align*}
In particular, $\left<   ^{t}\Lambda \left(u^{[1]}_{n}\left( \Lambda\right)\right),  B_{m+1}(x) \right>=0, m \geq n+1, \; n \geq 0.$ So, by Lemma \ref{lema1}, we get:
$$ ^{t}\Lambda \left(u^{[1]}_{n}\left( \Lambda\right)\right)= \sum_{\nu=0}^{n+1}  \lambda_{n, \nu} u_{\nu}, \; \geq 0,$$
with $\lambda_{n, \nu}=\left<^{t}\Lambda \left(u^{[1]}_{n}\left( \Lambda\right)\right), B_{\nu}(x) \right>, \; 0 \leq \nu \leq n+1.$
Notice that 
\begin{align*}
\lambda_{n, 0}&=\left<^{t}\Lambda \left(u^{[1]}_{n}\left( \Lambda\right)\right), B_{0}(x) \right>\\
&=\sum_{i=0}^{k}a_{i}(-1)^{i+1}\left< \left(Dx\right)^{i}D \left(  u^{[1]}_{n}\left( \Lambda\right)  \right) , 1  \right>  =0.  
\end{align*}
Also, if $1 \leq \nu \leq n$, then
\begin{align*}
& \lambda_{n, \nu}=\left<^{t}\Lambda \left(u^{[1]}_{n}\left( \Lambda\right)\right), B_{\nu}(x) \right>=\nu \left(\sum_{i=0}^{k}a_{i}\nu^{i}\right)\delta_{n, \nu-1} =0, \; \textrm{and}\\
&\lambda_{n, n+1}=\left<  ^{t}\Lambda \left(  u^{[1]}_{n} \left( \Lambda \right) \right), B_{n+1}(x)    \right>=\rho_{n} \neq 0.
\end{align*}
Consequently, $^{t}\Lambda \left(u^{[1]}_{n}\left( \Lambda\right)\right)=\lambda_{n, n+1} u_{n+1}.$
\end{proof}
\begin{proposition}\label{lambda-Appell-car}
Given a MPS $\{B_{n}\}_{n \geq 0}$, the following statements are equivalent.
\begin{description}
\item[a] $\{B_{n}\}_{n \geq 0}$ is $\Lambda$-Appell, that is, $B^{\left[1\right]}_{n}\left(x; \Lambda \right) = B_{n}(x)$;
\item[b] 
$ ^{t}\Lambda\left(u_{n} \right) = \rho_{n} u_{n+1},$ where $\rho_{n} = (n+1)\left(  \displaystyle \sum_{i=0}^{k}a_{i}(n+1)^{i}\right);$
\item[c] \label{lambda appell em formas}
$u_{n}=\left(n! \right) ^{-1}   \displaystyle \left(\prod_{s=1}^{n}\left( \sum_{i=0}^{k} a_{i} s^{i}  \right) \right)^{-1} \left( ^{t}\Lambda \right) ^{n} u_{0}.$
\end{description}
\end{proposition}

\begin{proof}
If $\{B_{n}\}_{n \geq 0}$ is $\Lambda$-Appell then $u_{n}=u^{[1]}_{n}\left(  \Lambda \right)$. Considering identity (\ref{lambda transpose un}), we obtain:
$$ ^{t}\Lambda \left(u_{n} \right)=\rho_{n}u_{n+1},\;\;\; \textrm{or}\;\;\;u_{n+1} = \rho_{n}^{-1}  \;^{t}\Lambda\left(u_{n} \right).$$
Recursively, we get that 
$u_{n} =\displaystyle \prod_{j=0}^{n-1} \rho_{j}^{-1} \left(^{t}\Lambda \right)^{n}\left(u_{0} \right).$ 
\newline
Since $\displaystyle \prod_{j=0}^{n-1} \rho_{j}^{-1} = \left(n! \right) ^{-1}   \displaystyle \left(\prod_{s=1}^{n}\left( \sum_{i=0}^{k} a_{i} s^{i}  \right) \right)^{-1}$, we obtain (c).

Conversely, if we suppose identity of item (b), then, from  (\ref{lambda transpose un}), we conclude that $ ^{t}\Lambda \left( u^{[1]}_{n}\left( \Lambda  \right)  \right)  =  \, ^{t}\Lambda  \left( u_{n} \right)$. Similarly, from identity of item (c), we can obtain the relation $u_{n+1} =  \rho_{n}^{-1}  \;^{t}\Lambda\left(u_{n} \right)$, as the next calculations explain, and again due to (\ref{lambda transpose un}), we conclude $ ^{t}\Lambda \left( u^{[1]}_{n}\left( \Lambda  \right)  \right)  =  \, ^{t}\Lambda  \left( u_{n} \right)$.
\begin{align*}
&u_{n+1}= \left((n+1)! \right) ^{-1}    \left(\prod_{s=1}^{n+1}\left( \sum_{i=0}^{k} a_{i} s^{i}  \right) \right)^{-1} \left( ^{t}\Lambda \right) ^{n+1} u_{0}\\
&= (n+1) ^{-1}    \left( \sum_{i=0}^{k} a_{i} (n+1)^{i} \right)^{-1} \left( ^{t}\Lambda \right)  \left(n! \right)^{-1} \left(\prod_{s=1}^{n}\left( \sum_{i=0}^{k} a_{i} s^{i}  \right) \right)^{-1}\left( ^{t}\Lambda \right) ^{n} u_{0}\\
&=(n+1)^{-1}    \left( \sum_{i=0}^{k} a_{i} (n+1)^{i} \right)^{-1}  \; ^{t}\Lambda \left(u_{n} \right).
\end{align*}
For any lowering operator $\mathcal{O}$, we can assure that $ ^{t}\mathcal{O} : \mathcal{P}^{\prime} \rightarrow \mathcal{P}^{\prime}$ is a one-to-one operator, and thus, in both situations, we conclude that $u^{[1]}_{n}\left(  \Lambda \right) = u_{n}$ proving that $\{B_{n}\}_{n \geq 0}$ is $\Lambda$-Appell.
\end{proof}
\section{A cubic decomposition of an Appell MPS}
For any MPS ${\{W_{n}\}}_{n \geq 0}$,
there are three MPSs ${\{P_{n}\}}_{n \geq 0}$, ${\{Q_{n}\}}_{n \geq
0}$ and ${\{R_{n}\}}_{n \geq 0}$, so that
\begin{align}
\label{decgeral1} W_{3n}(x) &= P_{n}(x^3)+xa_{n-1}^{1}(x^3)+x^2a_{n-1}^{2}(x^3) \\
\label{decgeral2} W_{3n+1}(x) &= b_{n}^{1}(x^3)+xQ_{n}(x^3)+x^2b_{n-1}^{2}(x^3) \\
\label{decgeral3} W_{3n+2}(x) &=
c_{n}^{1}(x^3)+xc_{n}^{2}(x^3)+x^2R_{n}(x^3),
\end{align}
with $\deg a_{n-1}^{1}\leq n-1$, $\deg a_{n-1}^{2}\leq n-1$, $\deg
b_{n}^{1}\leq n$, $\deg b_{n-1}^{2}\leq n-1$, $\deg c_{n}^{1}\leq n$, $\deg c_{n}^{2}\leq n$ and $a_{-1}^{1}(x)=a_{-1}^{2}(x)=b_{-1}^{2}(x)=0$.
This is a particular case of the general cubic
decomposition of any MPS presented in \cite{Mesquita articleJDE}, where all the parameters involved are considered zero. In this cubic decomposition (CD) (\ref{decgeral1})-(\ref{decgeral3})
of ${\{W_{n}\}}_{n \geq 0}$, the sequences:
\begin{itemize}
\item ${\{P_{n}\}}_{n \geq 0}, {\{Q_{n}\}}_{n \geq 0}, {\{R_{n}\}}_{n \geq 0}$ are called the principal components;
\item ${\{a_{n-1}^{1}\}}_{n \geq 0}$, ${\{a_{n-1}^{2}\}}_{n \geq 0}$, ${\{b_{n}^{1}\}}_{n \geq 0}$,
${\{b_{n-1}^{2}\}}_{n \geq 0}$, ${\{c_{n}^{1}\}}_{n \geq 0}$, ${\{c_{n}^{2}\}}_{n \geq 0}$ are called
the secondary components, since they are sequences of polynomials although not necessarily bases for the vector space of polynomials $\mathcal{P}$.
\end{itemize}
The nine component sequences are assembled in the following matrix \cite{Mesquita articleJDE}.
\begin{equation}\label{matrix M}
 M_{n}(x)=\left(
          \begin{array}{ccc}
            P_{n}(x) & a_{n-1}^{1}(x) & a_{n-1}^{2}(x) \\
            b_{n}^{1}(x) & Q_{n}(x) & b_{n-1}^{2}(x) \\
            c_{n}^{1}(x) & c_{n}^{2}(x) & R_{n}(x) \\
          \end{array}
        \right)
\end{equation}
\begin{lemma}\label{base dos 3}\cite{Mesquita articleJDE}
Let $P(x),\; Q(x)$ and  $R(x)$ be three polynomials. \par
$P(x^3)+xQ(x^3)+x^2R(x^3)=0$
implies  $P(x)=Q(x)=R(x)=0$.
\end{lemma}
\begin{proposition}\label{appell cd 9relations}
An Appell MPS $\{W_{n}\}_{n \geq 0}$ admits the CD
(\ref{decgeral1})-(\ref{decgeral3}) if and only if the following
relations are fulfilled for $n \geq 0$, where $I$ denotes the identity operator in $\mathcal{P}$.
\begin{align}
\label{appellCD-c1} \big(I+3xD \big)Q_{n}(x)&=(3n+1)P_{n}(x),\\
\label{appellCD-c2} \big(2I+3xD \big)b_{n-1}^{2}(x)&=(3n+1)a_{n-1}^{1}(x),\\
\label{appellCD-c3} 3D\,b_{n}^{1}(x)&=(3n+1)a_{n-1}^{2}(x),\\
\label{appellCD-c4}\big(I+3xD \big)c_{n}^{2}(x)&=(3n+2)b_{n}^{1}(x),\\
\label{appellCD-c5} \big(2I+3xD \big)R_{n}(x)&=(3n+2)Q_{n}(x),\\
\label{appellCD-c6} 3D\,c_{n}^{1}(x)&=(3n+2)b_{n-1}^{2}(x),\\
\label{appellCD-c7} \big(I+3xD \big)a_{n}^{1}(x)&=(3n+3)c_{n}^{1}(x),\\
\label{appellCD-c8} \big(2I+3xD \big)a_{n}^{2}(x)&=(3n+3)c_{n}^{2}(x),\\
\label{appellCD-c9} 3D\,P_{n+1}(x)&=(3n+3)R_{n}(x),
\end{align}
\end{proposition}
\begin{proof}
If $\{W_{n}\}_{n \geq 0}$ is an Appell MPS, then it fulfils $D\,W_{3n+1}(x)=(3n+1)W_{3n}(x),\; n \geq 0$. Considering the CD of $W_{3n+1}(x)$ and $W_{3n}(x)$, we obtain:
$$D\, \Big( b_{n}^{1}(x^3)+xQ_{n}(x^3)+x^2b_{n-1}^{2}(x^3)   \Big) =(3n+1) \Big(  P_{n}(x^3)+xa_{n-1}^{1}(x^3)+x^2a_{n-1}^{2}(x^3)\Big).$$
We can rewrite each term of the above relation as follows
\begin{align*}
&Q_{n}(x^3)+3x^3DQ_{n}(x^3)+x\Big( 2b_{n-1}^{2}(x^3)+3x^3Db_{n-1}^{2}(x^3)\Big)+3x^2D\,b_{n}^{1}(x^3)\\
&=(3n+1) \big(P_{n}(x^3)+xa_{n-1}^{1}(x^3)+x^2a_{n-1}^{2}(x^3)\big).
\end{align*}

Applying lemma \ref{base dos 3}, we get (\ref{appellCD-c1} -\ref{appellCD-c3}). The same procedure applied to 
\newline $D\,W_{3n+2}(x)$ $=(3n+2)W_{3n+1}(x)$ and $D\,W_{3n+3}(x)=(3n+3)W_{3n+2}(x)$ yields the remaining identities.\end{proof}
Let us remark that the nine conditions of Proposition \ref{appell cd 9relations} can be written using matrix identities, as follows.
\begin{equation}\label{I+3xD}
 \left(I+3x\,D\right)\left(\begin{array}{c}  c_{n-1}^{2}(x)\\ a_{n-1}^{1}(x) \\ Q_{n}(x)  \end{array}\right) = 
 \left(\begin{array}{ccc} 3n-1 & 0 & 0 \\ 0 & 3n & 0 \\ 0 & 0 & 3n+1 \end{array}\right) 
 \left(\begin{array}{c} b_{n-1}^{1}(x) \\ c^{1}_{n-1}(x) \\ P_{n}(x)\end{array}\right);
\end{equation}
\begin{equation}\label{2I+3xD}
\left(2I+3x\,D\right)\left(\begin{array}{c}  a_{n-1}^{2}(x)\\ b_{n-1}^{2}(x) \\ R_{n}(x)  \end{array}\right)= \left(\begin{array}{ccc} 3n & 0 & 0 \\ 0 & 3n+1 & 0 \\ 0 & 0 & 3n+2 \end{array}\right)  \left(\begin{array}{c} c_{n-1}^{2}(x) \\ a^{1}_{n-1}(x) \\ Q_{n}(x)\end{array}\right);
\end{equation}
\begin{equation}\label{3D}
 \left(3\,D\right)\left(\begin{array}{c}  b_{n}^{1}(x) \\ c^{1}_{n}(x)\\ P_{n+1}(x)  \end{array}\right)= \left(\begin{array}{ccc} 3n+1 & 0 & 0 \\ 0 & 3n+2 & 0 \\ 0 & 0 & 3n+3 \end{array}\right) \left(\begin{array}{c} a_{n-1}^{2}(x) \\ b^{2}_{n-1}(x) \\ R_{n}(x)\end{array}\right).
\end{equation}
\begin{proposition}\label{componentsequencesFAppell}
Let us consider an Appell MPS $\{W_{n}\}_{n \geq 0}$ and the correspondent CD
(\ref{decgeral1})-(\ref{decgeral3}). Then, for $n \geq 1$, we have:
\begin{align*}
a^{2}_{n-1}(x)&=\left((n+1)(3n+1)(3n+2)\right)^{-1} \mathcal{O}_{0,1,2}a^{2}_{n}(x),\\
b^{2}_{n-1}(x)&=\left((n+1)(3n+2)(3n+4)\right)^{-1} \mathcal{O}_{0,1,2}b^{2}_{n}(x),\\
 R_{n}(x)&=\left((n+1)(3n+4)(3n+5)\right)^{-1} \mathcal{O}_{0,1,2}R_{n+1}(x),\\
a^{1}_{n-1}(x)&=\left((n+1)(3n+1)(3n+2)\right)^{-1} \mathcal{O}_{2,0,1}a^{1}_{n}(x),\\
Q_{n}(x)&=\left((n+1)(3n+2)(3n+4)\right)^{-1} \mathcal{O}_{2,0,1}Q_{n+1}(x),\\
c^{2}_{n-1}(x)&=\left(n(3n+1)(3n+2)\right)^{-1} \mathcal{O}_{2,0,1}c^{2}_{n}(x),\\
P_{n}(x)&=\left((n+1)(3n+1)(3n+2)\right)^{-1} \mathcal{O}_{1,2,0}P_{n+1}(x),\\
b_{n-1}^{1}(x)&=\left(n(3n-1)(3n+1)\right)^{-1} \mathcal{O}_{1,2,0}b_{n}^{1}(x),\\
c_{n-1}^{1}(x)&=\left(n(3n+1)(3n+2)\right)^{-1} \mathcal{O}_{1,2,0}c_{n}^{1}(x),
\end{align*}
where $ \mathcal{O}_{0,1,2}=D\,\big(I+3x\,D\big)\,\big(2I+3x\,D\big)$,  $\mathcal{O}_{2,0,1}=\big(2I+3x\,D\big)\,D\,\big(I+3x\,D\big)$ and  $\mathcal{O}_{1,2,0}=\big(I+3x\,D\big)\,\big(2I+3x\,D\big)\,D$.
\end{proposition}
\begin{proof}
Let us consider (\ref{2I+3xD}), with $n \rightarrow n+1$. Applying $I+3x\,D$, we have:
\newline $\left(I+3x\,D\right) \left(2I+3x\,D\right)\left(\begin{array}{c}  a_{n}^{2}(x)\\ b_{n}^{2}(x) \\ R_{n+1}(x)  \end{array}\right)$
\newline $=\left(\begin{array}{ccc} 3n+3 & 0 & 0 \\ 0 & 3n+4 & 0 \\ 0 & 0 & 3n+5 \end{array}\right) (I+3x\,D) \left(\begin{array}{c} c_{n}^{2}(x) \\ a^{1}_{n}(x) \\ Q_{n+1}(x)\end{array}\right).$
\newline Attending to identity (\ref{I+3xD}), with $n \rightarrow n+1$, we get:
\newline $\left(I+3x\,D\right) \left(2I+3x\,D\right)\left(\begin{array}{c}  a_{n}^{2}(x)\\ b_{n}^{2}(x) \\ R_{n+1}(x)  \end{array}\right)$
\newline $=\left(\begin{array}{ccc} 3n+3 & 0 & 0 \\ 0 & 3n+4 & 0 \\ 0 & 0 & 3n+5 \end{array}\right) 
 \left(\begin{array}{ccc} 3n+2 & 0 & 0 \\ 0 & 3n+3 & 0 \\ 0 & 0 & 3n+4 \end{array}\right)  \left(\begin{array}{c} b_{n}^{1}(x) \\ c^{1}_{n}(x) \\ P_{n+1}(x)\end{array}\right).$
Applying $3\,D$ and using (\ref{3D}) we finally obtain:
\newline $(3\,D)\left(I+3x\,D\right) \left(2I+3x\,D\right)\left(\begin{array}{c}  a_{n}^{2}(x)\\ b_{n}^{2}(x) \\ R_{n+1}(x)  \end{array}\right)$
\newline $= \left(\begin{array}{c}  (3n+3)(3n+2)(3n+1)a_{n-1}^{2}(x)\\ (3n+4)(3n+3)(3n+2)b_{n-1}^{2}(x) \\ (3n+5)(3n+4)(3n+3)R_{n}(x)  \end{array}\right),$
\newline which correspond to the first three announced relations. The remaining six relations can be obtained in a similar way, considering, in the beginning of the procedure, identity (\ref{I+3xD}) and identity (\ref{3D}), respectively.
\end{proof}
In other words, the component sequences of a CD of an Appell sequence are also Appell with respect to certain operators. Furthermore, let us note that $\mathcal{O}_{0,1,2}$, $\mathcal{O}_{2, 0,1}$ and $\mathcal{O}_{1,2,0}$ define the operator related to the component sequences of the third, second and first column of $M_{n}(x)$ (\ref{matrix M}), respectively.
Taking into account identity (\ref{xD^2=DxD-D}) and Proposition \ref{Lambda-as-a-product}, the operators of Proposition \ref{componentsequencesFAppell} have the following alternative expression.
\begin{align*}
&\mathcal{O}_{0,1,2}=D\,\big(I+3X\,D\big)\,\big(2I+3X\,D\big)=2D+9DxD+9DxDxD,\\
&\mathcal{O}_{2,0,1}=\big(2I+3X\,D\big)\,D\,\big(I+3X\,D\big)= -D+9DxDxD, \\
& \mathcal{O}_{1,2,0}=\big(I+3X\,D\big)\,\big(2I+3X\,D\big)\,D= 2D-9DxD+9DxDxD.
\end{align*}
By Proposition \ref{lambda lowering operator}, these are lowering operators since the polynomials $f(x)=2+9x+9x^{2}$, $f(x)=-1+9x^{2}$ and $f(x)=2-9x+9x^{2}$ do not have positive integer roots. Since $2+9(n+1)+9(n+1)^{2} = (3n+4)(3n+5) $;  $-1+9(n+1)^{2} = (3n+2)(3n+4) $ and $2-9(n+1)+9(n+1)^{2} = (3n+1)(3n+2) $, and attending to the following identities given by Proposition \ref{componentsequencesFAppell}, we conclude that  the principal component sequence $\{R_{n}\}_{n \geq 0}$ is a $\mathcal{O}_{0,1,2}$-Appell sequence, $\{Q_{n}\}_{n \geq 0}$ is a $\mathcal{O}_{2,0,1}$-Appell sequence and $\{P_{n}\}_{n \geq 0}$ is a
$ \mathcal{O}_{1,2,0}$-Appell sequence.
$$R_{n}^{\left[1\right]}\left(x; \mathcal{O}_{0,1,2}  \right):=\left( (n+1)(3n+4)(3n+5) \right)^{-1} \mathcal{O}_{0,1,2}R_{n+1}(x) = R_{n}(x) ,$$
$$Q_{n}^{\left[1\right]}\left(x; \mathcal{O}_{2,0,1}  \right):=\left((n+1)(3n+2)(3n+4)\right)^{-1} \mathcal{O}_{2,0,1}Q_{n+1}(x) =Q_{n}(x) ,$$
$$P_{n}^{\left[1\right]}\left(x; \mathcal{O}_{1,2,0}  \right):=\left((n+1)(3n+1)(3n+2)\right)^{-1} \mathcal{O}_{1,2,0}P_{n+1}(x) = P_{n}(x).$$
With respect to the secondary components we achieve to the scenario described by the following result.
\begin{theorem}\label{secondarysequencesform}
Let us consider the secondary components of the CD (\ref{decgeral1})-(\ref{decgeral3})  of an Appell MPS. 
\newline Either the three sequences $\{ a_{n}^{1}\}_{n \geq 0}$, $\{ c_{n}^{1}\}_{n \geq 0}$ and $\{ b_{n}^{2}\}_{n \geq 0}$ are identical to the null sequence, or the aforementioned sequences are non-null, and in this case there are a positive integer $\kappa$ and three numerical sequences $\mu_{1,n},\; \mu_{2,n}$ and $\mu_{3,n}$ so that, for $n \geq 0$,
\begin{equation} \label{sec-comp-first-set}
b_{n+\kappa}^{2}(x) = \mu_{1,n} \hat{b}_{n}^{2}(x),\;\; a_{n+\kappa}^{1}(x) = \mu_{2,n} \hat{a}_{n}^{1}(x),\;\;c_{n+\kappa}^{1}(x) = \mu_{3,n} \hat{c}_{n}^{1}(x),
\end{equation}
where $\{ \hat{a}_{n}^{1}\}_{n \geq 0}$, $\{ \hat{c}_{n}^{1}\}_{n \geq 0}$ and $\{ \hat{b}_{n}^{2}\}_{n \geq 0}$ are monic polynomial sequences (MPSs).
\newline Either the three sequences $\{ a_{n}^{2}\}_{n \geq 0}$, $\{ c_{n}^{2}\}_{n \geq 0}$ and $\{ b_{n}^{1}\}_{n \geq 0}$ are identical to the null sequence, or the aforementioned sequences are non-null, and in this case there are a positive integer $\tau$ and three numerical sequences $\alpha_{1,n},\; \alpha_{2,n}$ and $\alpha_{3,n}$ so that, for $n \geq 0$,
\begin{equation} \label{sec-comp-second-set}
a_{n+\tau}^{2}(x) = \alpha_{1,n} \hat{a}_{n}^{2}(x),\;\; c_{n+\tau}^{2}(x) = \alpha_{2,n} \hat{c}_{n}^{2}(x),\;\; b_{n+\tau}^{1}(x) = \alpha_{3,n} \hat{b}_{n}^{1}(x),
\end{equation}
where $\{ \hat{a}_{n}^{2}\}_{n \geq 0}$, $\{ \hat{c}_{n}^{2}\}_{n \geq 0}$ and $\{ \hat{b}_{n}^{1}\}_{n \geq 0}$ are monic polynomial sequences (MPSs).
\newline Furthermore, there are two nonzero constants $b_{\kappa}^{2}$ and $a_{\tau}^{2}$ so that:
\begin{align}
\label{mu-1} &\mu_{1,n} = \frac{\left(\kappa +\frac{7}{3} \right)_{n}\left(\kappa +\frac{5}{3} \right)_{n}}{\left(\frac{4}{3} \right)_{n} \left(\frac{5}{3} \right)_{n}}
\left(
\begin{array}{c}
n+\kappa+1  \\
n
\end{array}
\right)b_{\kappa}^{2},\;\; n \geq 0;\\
\label{mu-2} &\mu_{2,n} = \frac{\left(\kappa +\frac{7}{3} \right)_{n-1}\left(\kappa +\frac{5}{3} \right)_{n}}{\left(\frac{4}{3} \right)_{n} \left(\frac{5}{3} \right)_{n-1}}
\left(
\begin{array}{c}
n+\kappa+1  \\
n
\end{array}
\right)b_{\kappa}^{2},\;\; n \geq 1;\\
\label{mu-3} &\mu_{3,n} = \frac{\left(\kappa +\frac{7}{3} \right)_{n-1}\left(\kappa +\frac{5}{3} \right)_{n}}{\left(\frac{4}{3} \right)_{n-1} \left(\frac{5}{3} \right)_{n-1}}
\left(
\begin{array}{c}
n+\kappa  \\
n
\end{array}
\right)\frac{b_{\kappa}^{2}}{\kappa+1},\;\; n \geq 1;\\
\label{alpha-1} &\alpha_{1,n} = \frac{\left(\tau +\frac{5}{3} \right)_{n}\left(\tau +\frac{4}{3} \right)_{n}}{\left(\frac{4}{3} \right)_{n} \left(\frac{5}{3} \right)_{n}}
\left(
\begin{array}{c}
n+\tau+1  \\
n
\end{array}
\right)a_{\tau}^{2},\;\; n \geq 0;\\
\label{alpha-2}  &\alpha_{2,n} = \frac{\left(\tau+\frac{5}{3} \right)_{n}\left(\tau +\frac{4}{3} \right)_{n}}{\left(\frac{4}{3} \right)_{n} \left(\frac{5}{3} \right)_{n-1}}
\left(
\begin{array}{c}
n+\tau  \\
n
\end{array}
\right)\frac{a_{\tau}^{2}}{\tau+1},\;\; n \geq 1;\\
\label{alpha-3}  &\alpha_{3,n} = \frac{\left(\tau +\frac{5}{3} \right)_{n-1}\left(\tau +\frac{4}{3} \right)_{n}}{\left(\frac{4}{3} \right)_{n-1} \left(\frac{5}{3} \right)_{n-1}}
\left(
\begin{array}{c}
n+\tau  \\
n
\end{array}
\right)\frac{a_{\tau}^{2}}{\tau+1},\;\; n \geq 1;
\end{align}
\newline with $\mu_{2,0} =\frac{2 b_{\kappa}^{2}}{3\left(\kappa+\frac{4}{3}\right)}$;
$ \mu_{3,0} =\frac{2 b_{\kappa}^{2}}{9(\kappa+1)\left(\kappa+\frac{4}{3}\right)}$;
$\alpha_{2,0} =\frac{2 a_{\tau}^{2}}{3\left(\tau+1\right)}$;
$ \alpha_{3,0} =\frac{2 a_{\tau}^{2}}{9(\tau+1)\left(\tau+\frac{2}{3}\right)}$;
\newline and where $\left( a \right)_{n} $ is the Pochhammer symbol for the falling factorial 
\newline $\left( a \right)_{n} = a\left( a+1\right) \cdots \left( a + n-1\right),$ $\; n \geq 1,$ $\; \left( a \right)_{0} = 1$.\end{theorem}
\begin{proof}
Let us suppose that $\{ b_{n}^{2}\}_{n \geq 0}$ is a non-null sequence and let $\kappa$ be the smallest index  such that $b_{\kappa}^{2}(x) \neq 0$ and $b_{n}^{2} = 0, \;\; 0 \leq n \leq \kappa -1$, if $\kappa \geq 1$. Thus, due to (\ref{appellCD-c2}), 
$a_{n}^{1} = 0, \;\; 0 \leq n \leq \kappa -1$ and hence by (\ref{appellCD-c7}), $c_{n}^{1} = 0, \;\; 0 \leq n \leq \kappa -1$. According to (\ref{appellCD-c6}), $c_{\kappa}^{1}(x)=c_{\kappa}^{1}$ (constant) and in virtue of (\ref{appellCD-c7}) $a_{\kappa}^{1}(x)=a_{\kappa}^{1}$ (constant). 
Similarly, (\ref{appellCD-c2}) implies $b_{\kappa}^{2}(x)=b_{\kappa}^{2} \neq 0$ (nonzero constant) and consequently $a_{\kappa}^{1} \neq 0$ and $c_{\kappa}^{1} \neq 0$. In particular,
$\{ a_{n}^{1}\}_{n \geq 0}$ and $\{ c_{n}^{1}\}_{n \geq 0}$ are also non-null sequences.
Taking into consideration Proposition \ref{O-Appell sequences structure} and Proposition \ref{componentsequencesFAppell}, we conclude (\ref{sec-comp-first-set}). 
\newline Indeed, by virtue of Proposition \ref{appell cd 9relations} relations, the relevant fact is that if
 exists a positive integer $n_{0}$ such that $a_{n_{0}}^{1}=0$, then $a_{n_{0-1}}^{1}=0$ and likewise for $\{ b_{n}^{2}\}_{n \geq 0}$ and $\{ c_{n}^{1}\}_{n \geq 0}$. 
\newline The second sentence (\ref{sec-comp-second-set}) is explained in the same manner, by supposing the smallest index $\tau$ such that $a_{\tau}^{2}(x) \neq 0$ and $a_{n}^{2} = 0, \;\; 0 \leq n \leq \tau -1$, if $\tau \geq 1$, and arguing with identities (\ref{appellCD-c3}), (\ref{appellCD-c4}) and (\ref{appellCD-c8}).   
\newline
With regard to the initial conditions and following the indicated notation, we begin to set $\mu_{1,0} = b_{\kappa}^{2}$ and $\alpha_{1,0} = a_{\tau}^{2}$. From (\ref{appellCD-c2}) and (\ref{appellCD-c8}), we have $\mu_{2,0} = \frac{2}{3\kappa +4}b_{\kappa}^{2}$ and $\alpha_{2,0} =  \frac{2}{3\tau +3}a_{\tau}^{2}$. From  (\ref{appellCD-c7}) and (\ref{appellCD-c4}), we have $\mu_{3,0} = \frac{2}{(3\kappa+3)(3\kappa +4)}b_{\kappa}^{2}$ and $\alpha_{3,0} =  \frac{2}{(3\tau+2)(3\tau +3)}a_{\tau}^{2}$.
\newline
Inserting the expressions of (\ref{sec-comp-first-set}) in identities (\ref{appellCD-c2}), (\ref{appellCD-c6}) and (\ref{appellCD-c7}) - all three with an appropriate translation of the index $n$ - and comparing the leading coefficients of both sides, we obtain the following system.
\begin{align*}
&\left(2+3n\right)\mu_{1,n} = \left(3n+3\kappa+4\right)\mu_{2,n},\\
&3\left(n+1\right)\mu_{3,n+1} = \left(3n+3\kappa+5\right)\mu_{1,n},\\
&\left(1+3n\right)\mu_{2,n} = \left(3n+3\kappa+3\right)\mu_{3,n}.
\end{align*}
Its resolution yields $\displaystyle \mu_{3,n+1} = \frac{\left(n + \kappa +1\right)\left(n + \kappa +\frac{4}{3}\right)\left(n + \kappa +\frac{5}{3}\right) }{ \left(n+\frac{1}{3}\right)\left(n+\frac{2}{3}\right)\left(n+1\right)}\mu_{3,n}$ and we achieve to identity (\ref{mu-3}), without forgetting the initial conditions. The expressions of  $\mu_{1,n}$ and $\mu_{2,n}$ can be calculated looking again to the previous system and using (\ref{mu-3}). Equalities (\ref{alpha-1}), (\ref{alpha-2}) and (\ref{alpha-3}) are obtained analogously.
\end{proof}
\begin{remark}
Regarding the proof of the previous Theorem \ref{secondarysequencesform}, we could have pursuit without using Proposition \ref{O-Appell sequences structure}, by strictly reasoning with the relations of Proposition \ref{appell cd 9relations}. Nevertheless, by this option, we aim to remark that the underlined feature can be put as a general characteristic valid for any sequence fulfilling assumptions of Proposition \ref{O-Appell sequences structure}. 
\end{remark}
\section{ $\Lambda$-Appell orthogonal sequences for a specific $\Lambda$ (k=2)}\label{Appell CD again}
Considering $\Lambda=a_{0}D+a_{1}DxD+a_{2}\left(Dx\right)^2D$, we obtain from (\ref{Bn[1] definition})
\begin{equation}\label{Bn[1] definition for 3termslamba}
B^{\left[1\right]}_{n}\left(x; \Lambda \right)=\frac{\left( \Lambda B_{n+1}\right) (x)}{\rho_{n}},\, n \geq 0,\;\; \textrm{with}
\end{equation}
\begin{equation}
\rho_{n}=(n+1)\left(a_{0}+a_{1}(n+1)+a_{2}(n+1)^{2} \right) \neq 0, \, n \geq 0.
\end{equation}
Given a $\Lambda$-Appell MPS, we already indicated in Proposition \ref{lambda-Appell-car} some functional relations which translate that characteristic, as for instance, $^{t}\Lambda\left(u^{[1]}_{n}\right)=\rho_{n}u_{n+1}$. Adding more functional ingredients to the hypotheses set, we have the possibility to better acquaint specific MPSs. Theorem \ref{there-are-not} establishes the non-existence of orthogonal $\Lambda$-Appell sequences and suggests a track towards the resolution of further questions concerning the $\Lambda$-Appell MPSs. Its proof requires the relations of Lemma \ref{Lambda-pu}.
\begin{lemma}\label{Lambda-pu}
Considering $k=2$, let $\Lambda=a_{0}D+a_{1}DxD+a_{2}\left(Dx\right)^2D$, for some constants $a_{0}, \,a_{1}$ and $a_{2}$, and let $f ,\,p\in \mathcal{P}$ and $u \in  \mathcal{P}^{\prime}$. Then we have:
\begin{align}
&\label{lambda fp} \Lambda\left( f\,p \right) = f \Lambda\left( p \right) + p \Lambda\left( f \right) + 2 \left(a_{1} + 3a_{2} \right) x f^{\prime}p^{\prime}\\
\nonumber &\hspace{1.3cm} +3 a_{2} x^{2} f^{\prime\prime} p^{\prime} +3a_{2} x^2 f^{\prime}p^{\prime\prime},\\
&\label{Tlambda fu}   ^{t}\Lambda\left(fu\right)=f \,^{t}\Lambda(u)-\Lambda(f)u+2a_{1}\left(f^{\prime}+ xf^{(2)} \right)u\\
\nonumber &\hspace{1.3cm}+\Big(  2(a_{1}-3a_{2})xf^{\prime}-3a_{2}x^2f^{(2)}\Big) u^{\prime}-3a_{2}x^2f^{\prime}u^{(2)}.
\end{align}
\end{lemma}
\begin{proof}
Recall that from (\ref{lambda withk=2 in xD}), we have $\Lambda = \left( a_{0}+a_{1}+a_{2} \right)D+\left(a_{1}+3a_{2} \right)xD^{2} + a_{2} x^{2} D^{3}$, so that a straightforward calculation yields identity (\ref{lambda fp}).
\newline
On the other hand,  if $p \in \mathcal{P}$,

\begin{align*}
 & \langle ^{t}\Lambda\left(fu\right), p \rangle = \langle fu, \Lambda(p) \rangle = \langle u, f\Lambda(p) \rangle\\
 & \stackrel{(\ref{lambda fp})}{=}  \langle u, \Lambda(fp) \rangle - \langle  u, \Lambda(f) p \rangle -2(a_{1}+3a_{2}) \langle u , x f^{\prime} p^{\prime} \rangle \\
 &  \hspace{0.9cm} -3a_{2} \langle u, x^{2} f^{\prime\prime} p^{\prime} \rangle- 3a_{2} \langle u, x^{2} f^{\prime} p^{\prime\prime} \rangle\\
&  \stackrel{(\ref{funcionalDu})}{=} \langle f \,^{t} \Lambda u, p \rangle - \langle \Lambda(f) u, p \rangle + 2(a_{1}+3a_{2}) \langle D\left(f^{\prime}xu\right) , p \rangle \\
& \hspace{0.9cm} + 3a_{2} \langle D\left (f^{\prime\prime} x^{2} u \right),   p \rangle - 3a_{2} \langle D^{2} \left(f^{\prime}x^{2}u\right),   p \rangle.
\end{align*}
We can now extend the three latest operators $D\left(f^{\prime}xu\right)$, $D\left (f^{\prime\prime} x^{2} u \right)$ and $ D^{2} \left(f^{\prime}x^{2}u\right)$ following the product rule (\ref{derivada do produto}),
\begin{align*}
&D\left(f^{\prime}xu\right) = \left( f^{\prime\prime} x + f^{\prime}  \right) u + f^{\prime} x u^{\prime}\\
&D\left (f^{\prime\prime} x^{2} u \right) = f^{(3)}x^{2}u+2 f^{\prime\prime} x u + f^{\prime\prime} x^{2} u^{\prime}\\
& D^{2} \left(f^{\prime}x^{2}u\right) = \left( f^{(3)} x^{2} + 4f^{\prime\prime} x + 2 f^{\prime}\right)u + \left( 2 f^{\prime\prime} x^{2} + 4 f^{\prime} x  \right) u^{\prime}  +f^{\prime}x^{2}u^{\prime\prime}
\end{align*}
and hence, the identity
\newline $^{t}\Lambda\left(fu\right) = f \,^{t} \Lambda u -  \Lambda(f) u + 2(a_{1}+3a_{2}) D\left(f^{\prime}xu\right)  + 3a_{2} D\left (f^{\prime\prime} x^{2} u \right) - 3a_{2} D^{2} \left(f^{\prime}x^{2}u\right)$
will justify (\ref{Tlambda fu}), after the above substitutions.
\end{proof}
\begin{theorem}\label{there-are-not}
Let $\Lambda=a_{0}D+a_{1}DxD+a_{2}\left(Dx\right)^2D$, for some constants $a_{0}, \,a_{1}$ and $a_{2}$, such that  $a_{2} \neq 0$ and the polynomial $f(x)=a_{0}+a_{1}x+a_{2}x^{2}$ does not have positive integer roots. There are not $\Lambda$-Appell orthogonal sequences.
\end{theorem}
\begin{proof} Let us consider a MPS $\{B_{n}\}_{n \geq 0}$ $\Lambda$-Appell and orthogonal. Therefore, $\left( \Lambda B_{n+1}\right) (x) = \rho_{n} B_{n}(x)$, and from Proposition \ref{tLambda an[1]} and Lemma \ref{The transpose operator tLambda}, we have:
\begin{align} \label{lambda-Appell-car-1}
& ^{t}\Lambda\left(u_{n}\right)=\rho_{n}u_{n+1},
\;\;\textrm{with}\; \;^{t}\Lambda=-a_{0}D+a_{1}DxD-a_{2}\left(Dx\right)^2D.
\end{align}
By Theorem \ref{regular}, the orthogonality of  $\{B_{n}\}_{n \geq 0}$ corresponds to the functional identity: 
\begin{equation}\label{ortho-functional-car}
u_{n}=\Big(<u_{0},B_{n}^{2}>\Big)^{-1}B_{n}u_{0},\;\;\;n \geq 0.
\end{equation}
Summing up the two previous relations, we gain the following one.
\begin{equation}\label{appell+ortho-functional}
^{t}\Lambda\left(B_{n}u_{0}\right)=\lambda_{n}B_{n+1}u_{0},\;\;\; \textrm{with}\;\;\lambda_{n} = \rho_{n}\frac{ \langle u_{0}, B_{n}^{2}(x) \rangle}{ \langle u_{0}, B_{n+1}^{2}(x) \rangle} = \frac{\rho_{n}}{\gamma_{n+1}} .
\end{equation}
Let us consider (\ref{appell+ortho-functional}) with $n=0$. Taking into account relation (\ref{lambda withk=2 in xD}), we establish the first functional equation in terms of $u_{0}$, as follows. We also notice that $\gamma_{1} = \langle u_{0},B_{1}^{2} \rangle$, by (\ref{betas e gammas}).
\begin{equation}
\label{TLambda-u0=b1-u0} ^{t}\Lambda\left(u_{0}\right)=\frac{\rho_{0}}{\gamma_{1}} B_{1}(x) u_{0},
\end{equation}
\par $ \Leftrightarrow \Big( \left(-a_{0}+a_{1}-a_{2} \right)D+\left(  a_{1} -3a_{2} \right)xD^{2}-a_{2}x^{2} D^{3} \Big)u_{0}=\frac{\rho_{0}}{\gamma_{1}} B_{1}(x) u_{0}$
\begin{equation}
\label{u0-third-order-1} \Leftrightarrow a_{2} x^{2} u_{0}^{(3)}  - \left(a_{1}-3a_{2}\right)x u_{0}^{(2)} + \left(a_{0}-a_{1}+a_{2} \right)u_{0}^{\prime} + \frac{\rho_{0}}{\gamma_{1}} B_{1}(x) u_{0}=0.
\end{equation}
Returning to identity (\ref{appell+ortho-functional}), its left-hand can be expanded by means of 
(\ref{Tlambda fu}).  Inserting also the content of  (\ref{TLambda-u0=b1-u0}) we get the next general identity, where $n \geq 0$.
\begin{align} \label{1.12}
-3a_{2}x^{2}B_{n}^{\prime}u_{0}^{\prime\prime} &+ \Big( 2\left( a_{1}-3a_{2}\right) x\,B_{n}^{\prime}   -3a_{2} x^{2}\,B_{n}^{\prime\prime}\Big) u_{0}^{\prime} \\
\nonumber &= \Big( \lambda_{n}B_{n+1}-2 a_{1} \left( x B_{n}^{\prime\prime} + B_{n}^{\prime} \right) -  \lambda_{0} B_{1}B_{n} + \Lambda(B_{n}) \Big)u_{0}
\end{align}
Let us take identity (\ref{1.12}) with $n=1$ and recall that $\Lambda(B_{1})= \rho_{0} = a_{0}+a_{1}+a_{2}$.
\begin{align} \label{1.13 ou 3.2}
-3a_{2}x^{2}u_{0}^{\prime\prime} + 2\left( a_{1}-3a_{2}\right) x u_{0}^{\prime} &= U(x) u_{0},\\
\label{U(x)} \textrm{where}\;\;U(x)&=  \lambda_{1}B_{2} - \lambda_{0} B_{1}^{2} + a_{0} - a_{1} +a_{2}.
\end{align}
Using  (\ref{1.13 ou 3.2}) we can eliminate $u_{0}^{\prime\prime}$ in (\ref{1.12}) reducing it to the next relation.
\begin{align} \label{1.14}
&-3a_{2} x^{2}\,B_{n}^{\prime\prime} u_{0}^{\prime} \\
\nonumber 
&=\Big( \lambda_{n}B_{n+1}-  \lambda_{0} B_{1}B_{n} + \Lambda(B_{n}) -2 a_{1} x B_{n}^{\prime\prime} - B_{n}^{\prime} \left( \lambda_{1} B_{2} -  \lambda_{0} B_{1}^{2} + \rho_{0}       \right)  \Big)u_{0}
\end{align}
From this latest, with $n=2$, we can write the following first order equation on $u_{0}$, remarking that $\Lambda(B_{2}) = \rho_{1} B_{1}$.
\begin{align} \label{1.15 ou 3.3}
&-6a_{2} x^{2}u_{0}^{\prime} = V(x) u_{0},\;  \textrm{where}\\ 
\label{V(x)} & V(x) = \lambda_{2}B_{3} - \lambda_{0} B_{1}B_{2} + \rho_{1} B_{1} - 4 a_{1} x - B_{2}^{\prime} \left( \lambda_{1} B_{2} - \lambda_{0} B_{1}^{2} + \rho_{0} \right).
\end{align}
The three equations (\ref{u0-third-order-1}), (\ref{1.13 ou 3.2}) and  (\ref{1.15 ou 3.3}) constitute a system on $u_{0}$ that can be reduced to a list of three equations of first order by systematic derivation and elimination of the term of higher order. The final list is:
\begin{align} 
\nonumber -6a_{2} x^{2}u_{0}^{\prime} &= V(x) u_{0},\\ 
\label{3.7}  \Big( 4a_{1} x+V(x) \Big) u_{0}^{\prime}  &= \Big( 2 U(x)-V^{\prime}(x) \Big)u_{0},\\
\label{3.8}   x W_{1}(x) u_{0}^{\prime}  &= W_{2}(x) u_{0},\;\;\; \textrm{where}:
\end{align}
\begin{align} 
\label{W1} &  W_{1}(x) = -2\left( a_{1}-3a_{2}\right)\left( a_{1}+3a_{2}\right) + 3a_{2} \big( 3a_{0}-a_{1} - 3a_{2} - U(x)\big);\\
\label{W2} & W_{2}(x) = -\left( a_{1}+3a_{2}\right)U(x) +3a_{2} x\left( -3\lambda_{0} B_{1}(x) +U^{\prime}(x)\right).
\end{align}
The second step of elimination consists of eliminating $u_{0}^{\prime}$ between the three possible pairs of the above list, resulting three identities of the form $p\,u_{0}=0$ ($p \in \mathcal{P}$), which by Lemma \ref{phi-u=0} provide the following list of polynomial relations. 
\begin{align} 
\label{3.12} &V^{2}(x) = x \Big( 6a_{2} x \big(V^{\prime}(x) - 2U(x) \big) -4a_{1}V(x) \Big)\\
\label{3.13}  & W_{1}(x) V(x) + 6a_{2} x W_{2}(x) = 0,\\
\label{3.14} &V(x)W_{2}(x) = x \Big( W_{1}(x)\big(2U(x) - V^{\prime}(x) \big)   -4a_{1}W_{2}(x)\Big)
\end{align}
The next reasonings are based on the definitions of the polynomials $U(x),$ $V(x),$ $W_{1}(x)$ and $W_{2}(x)$ and on a comparative analysis of the degrees of each side of an equality. 
Identity  (\ref{3.12}) implies that the leading coefficient of $V(x)$ must be zero, that is,
\begin{equation} \label{c1}
\lambda_{0} - 2 \lambda_{1} + \lambda_{2} = 0. 
\end{equation}
Proceeding with a similar approach to identity (\ref{3.14}) and taking into consideration that $\lambda_{2} = -\lambda_{0} + 2 \lambda_{1}$, we conclude that the leading coefficient of the right-hand of  (\ref{3.14}) must vanish, which implies
\begin{equation} \label{c2}
\lambda_{0} = \lambda_{1}. 
\end{equation}
Henceforth,  $\lambda_{0} = \lambda_{1}= \lambda_{2}$. Turning our attention back to  (\ref{3.12}), and adding the recent simplifications on the parameters, we obtain:
\begin{align} 
\label{c3} &-\lambda_{0}\left(\beta_{0}-2\beta_{1}+\beta_{2} \right) = 0,\\
\label{c4} & 12a_{2}\left( \beta_{1}-\beta_{0}\right) \lambda_{0} = 0.
\end{align}
As a consequence, $\beta_{0}=\beta_{1}=\beta_{2}$. Finally, introducing all this information on the polynomial of identity  (\ref{3.13}), we do not find the null polynomial; instead, the coefficient of $x^3$, for instance, is $-54a_{2}^{2}\lambda_{0}$ that cannot be null under the set of hypotheses taken, where $a_{2} \neq 0$ and $\lambda_{0} = \frac{\rho_{0}}{\gamma_{1}} \neq 0$. This contradiction ends the proof.
\end{proof}

\begin{remark}
As an immediate corollary of Theorem \ref{there-are-not}, we must mention that the principal components of the CD (\ref{decgeral1})-(\ref{decgeral3}) of an Appell sequence are not orthogonal.
\end{remark}

Looking up to the operator $\Lambda$ restrictively, we recall that when $a_{1}=a_{2}=0$, the $\Lambda$-Appell orthogonal sequences are reduced to the single Hermite sequence and when $a_{2}=0$, we have the result obtained in \cite{QuadraticAppell}. Nonetheless, in this last reference, the operator used is described by the choice $a_{0}= \epsilon$ and $a_{1}= 2$ and the techniques used in the proof of Theorem \ref{there-are-not} are nicely adapted to prove an analogous result for $\Lambda=a_{0}D+a_{1}DxD$ reinforcing what is already known.

\begin{proposition}
Let $\Lambda=a_{0}D+a_{1}DxD$, for some constants $a_{0}$ and $a_{1}$, such that $a_{1} \neq 0$ and the polynomial $f(x)=a_{0}+a_{1}x$ does not have positive integer roots. The $\Lambda$-Appell orthogonal sequences are the Laguerre sequences with parameter $\displaystyle \alpha=\frac{a_{0}}{a_{1}}$, up to an affine transformation.
\end{proposition}
\begin{proof}
Let us consider $a_{2}=0$ in the proof of Theorem \ref{there-are-not}, and in particular, the three equations (\ref{u0-third-order-1}), (\ref{1.13 ou 3.2}) and  (\ref{1.15 ou 3.3}) are now the following.
\begin{align} 
\nonumber - a_{1} x u_{0}^{(2)} + \left(a_{0}-a_{1}\right)u_{0}^{\prime} + \lambda_{0} B_{1}(x) u_{0}&=0,\\
\nonumber 2a_{1} x u_{0}^{\prime}  &= U(x) u_{0},\\
\nonumber  V(x) &= 0.
\end{align}
The consecutive derivation and elimination of the terms of higher order establishes the next system.
\begin{align} 
\nonumber 2a_{1} x u_{0}^{\prime}  &= U(x) u_{0},\\
\nonumber  V(x) &= 0,\\
\nonumber  T(x) = \big(2a_{0}-U(x)\big) U(x) + 2 a_{1}x\left( 2\lambda_{0} B_{1}(x)    -U^{\prime}(x) \right) &=0. 
\end{align}
The conditions $V(x)=0$ and $T(x)=0$ correspond to the next list of five identities, where $\beta_{0} \left(\beta_{0}- \beta_{1}\right) \neq 0$ (otherwise, we would have $a_{1}=0$),  and it is assured that $a_{0}+a_{1} \neq 0$ and $a_{0}+2a_{1} \neq 0$, because $f(x)$ does not have positive integer roots.
\begin{align} 
\nonumber  &\lambda_{2} = 2\lambda_{1}-\lambda_{0},\hspace{0.5cm} \gamma_{2} = \frac{2(a_{0}+2a_{1})\gamma_{1}}{a_{0}+a_{1}}, \hspace{0.5cm} \beta_{2} = 2\beta_{1} - \beta_{0},\\
\nonumber & 2\gamma_{1} = -\beta_{0}\left(\beta_{0}- \beta_{1}\right),
\hspace{0.5cm} \beta_{1} = \frac{3a_{1}+a_{0}}{a_{0}+a_{1}}\beta_{0}.
\end{align}
Under these assumptions, the functional condition $2a_{1} x u_{0}^{\prime}  = U(x) u_{0}$ becomes a very familiar one, more precisely,
\begin{align} 
&D\left( x u_{0} \right) + \Psi(x)u_{0} = 0,\;\;\;\textrm{where}\;\;\; \Psi(x) = \frac{a_{0}+a_{1}}{\beta_{0}a_{1}}x - \frac{a_{0}}{a_{1}} - 1.
\end{align}
Implementing the affine transformation $\displaystyle \frac{\beta_{0}a_{1}}{a_{0}+a_{1}}x$, the resultant sequence is the Laguerre sequence with parameter $\alpha = \displaystyle  \frac{a_{0}}{a_{1}}$, being sure that $\alpha \notin \mathbb{Z}^{-}$.
\end{proof}

As a last remark, we must add that a particular $\Lambda$ operator with $k=3$, so-called $\mathcal{G}_{\epsilon, \mu}$, was debated in \cite{QuadAppell-two}. It was proved the absence of orthogonal $\mathcal{G}_{\epsilon, \mu}$-Appell sequences indicating that Theorem  \ref{there-are-not} may be generalized at a latter time.


\end{document}